\documentclass[11pt]{amsart}
\usepackage[T1]{fontenc}
\usepackage[latin9]{inputenc}
\usepackage{verbatim}
\usepackage{amsthm}
\usepackage{amstext}
\usepackage{amssymb}
\usepackage{esint}

\makeatletter
\numberwithin{equation}{section}
\numberwithin{figure}{section}

\advance\hoffset by -0.5 truecm
\usepackage{amsthm}\usepackage{amstext}

\pdfpageheight\paperheight
\pdfpagewidth\paperwidth

\numberwithin{equation}{section}
\numberwithin{figure}{section}
\theoremstyle{plain}
\newtheorem{thm}{Theorem}[section]
  \theoremstyle{plain}
  \newtheorem{assumption}[thm]{Assumption}
  \theoremstyle{plain}
  \newtheorem{lem}[thm]{Lemma}
\theoremstyle{plain}
  \newtheorem{prop}[thm]{Proposition}
  \theoremstyle{remark}
  \newtheorem{rem}[thm]{Remark}
  \theoremstyle{plain}
  \newtheorem{cor}[thm]{Corollary} 
 \theoremstyle{definition}
  \newtheorem{defn}[thm]{Definition}
  \theoremstyle{remark}
  \newtheorem*{acknowledgement*}{Acknowledgement}

\usepackage{times}
\usepackage{bbold}

\makeatother

\begin{document}

\title{Index computations in Rabinowitz Floer homology}

\author{Will J. Merry and Gabriel P. Paternain}

\address{Department of Pure Mathematics and Mathematical Statistics, University
of Cambridge, Cambridge CB3 0WB, UK}

\email{\texttt{w.merry@dpmms.cam.ac.uk, g.p.paternain@dpmms.cam.ac.uk}}
\begin{abstract}
In this note we study two index questions. In the first we establish
the relationship between the Morse indices of the free time action
functional and the fixed time action functional. The second is related
to Rabinowitz Floer homology. Our index computations are based on
a correction term which is defined as follows: around a non-degenerate
Hamiltonian orbit lying in a fixed energy level a well-known theorem
says that one can find a whole cylinder of orbits parametrized by
the energy. The correction term is determined by whether the periods
of the orbits are increasing or decreasing as one moves up the orbit
cylinder. We also provide an example to show that, even above the
Ma\~n\'e critical value, the periods may be increasing thus producing
a jump in the Morse index of the free time action functional in relation
to the Morse index of the fixed time action functional. 
\end{abstract}
\maketitle

\centerline{\it Dedicated to Richard S. Palais on the occasion of his 80th
birthday.}

\section{Introduction}

Let $(M,g)$ denote a closed connected orientable Riemannian manifold
with cotangent bundle $\pi:T^{*}M\rightarrow M$. Let $\omega_{0}=d\lambda_{0}$
denote the canonical symplectic form $dp\wedge dq$ on $T^{*}M$,
where $\lambda_{0}$ is the Liouville $1$-form. Let $\widetilde{M}$
denote the universal cover of $M$. Let $\sigma\in\Omega^{2}(M)$
denote a closed \textbf{weakly exact} $2$-form, by this we mean that
the pullback $\widetilde{\sigma}\in\Omega^{2}(\widetilde{M})$ is
exact. We assume in addition that $\widetilde{\sigma}$ admits a \textbf{bounded}
primitive. This means that there exists $\theta\in\Omega^{1}(\widetilde{M})$
with $d\theta=\widetilde{\sigma}$, and such that \[
\left\Vert \theta\right\Vert _{\infty}:=\sup_{q\in\widetilde{M}}\left|\theta_{q}\right|<\infty,\]
 where $\left|\cdot\right|$ denotes the lift of the metric $g$ to
$\widetilde{M}$. Let \[
\omega:=\omega_{0}+\pi^{*}\sigma\]
denote the \textbf{twisted symplectic form}\emph{ }determined by $\sigma$.
We call the symplectic manifold $(T^{*}M,\omega)$ a \textbf{twisted
cotangent bundle}.{ In fact, for the purposes of this paper the assumptions
that $\sigma$ is weakly exact and admits a bounded primitive are
not strictly necessary; see Section \ref{sub:Dropping-the-assumption}
below.}

Let us fix once and for all a \textbf{Tonelli Hamiltonian }$H\in C^{\infty}(T^{*}M,\mathbb{R})$.
Here we recall that the classical Tonelli assumption means that $H$
is \textbf{fibrewise strictly convex }and \textbf{superlinear}. In
other words, the Hessian $\nabla^{2}H$ of $H$ restricted to each
tangent space $T_{q}^{*}M$ is positive definite, and \[
\lim_{\left|p\right|\rightarrow\infty}\frac{H(q,p)}{\left|p\right|}=\infty\]
uniformly for $q\in M$. Let $X_{H}$ denote the symplectic gradient
of $H$ with respect to $\omega$, and let $\phi_{t}:T^{*}M\rightarrow T^{*}M$
denote the flow of $X_{H}$. Denote by $L\in C^{\infty}(TM,\mathbb{R})$
the \textbf{Fenchel dual Lagrangian}. This is the unique Lagrangian
$L$ on $TM$ defined by \[
L(q,v):=\max_{p\in T_{q}^{*}M}\{p(v)-H(q,p)\}.\]
The Lagrangian $L$ is also fibrewise strictly convex and superlinear. 

\begin{assumption} \emph{\label{ass:nondegneracy}Fix now once and
for all a regular value $k$ of $H$, and denote by $\Sigma:=H^{-1}(k)$.
Thus $\Sigma$ is a closed hypersurface in $T^{*}M$. Since $H$ is
autonomous, $\phi_{t}$ preserves $\Sigma$. We shall consider orbits
$y:\mathbb{R}/T\mathbb{Z}\rightarrow\Sigma$ of $X_{H}$ that admit
}\textbf{\emph{non-degenerate orbit cylinders}}\emph{.} \end{assumption}\emph{ }

The precise definition of a non-degenerate orbit cylinder is given
below in Definition \ref{orbit cylinder}, but intutively we ask that
$y$ can be included in a whole family of orbits parametrized by energy,
with the property that the periods are either increasing or decreasing
(i.e. not constant) as one moves up the orbit cylinder. The key setting
we have in mind where this condition is satisfied is when the orbit
$y$ is\textbf{ strongly non-degenerate}, by which we mean that the
\textbf{nullity}\emph{ }of $y$, $\nu(y)$ satisfies\begin{equation}
\nu(y):=\dim\,\ker(d_{y(0)}\phi_{T}-\mathbb{1})=1.\label{eq:snd}\end{equation}
In this case $y$ admits a unique non-degenerate orbit cylinder (cf.
the discussion after Definition \ref{orbit cylinder} below). For
an explanation as to the terminology {}``strongly non-degenerate''
see Remark \ref{rem:A-remark-on terminiology} below.\newline

In this note we address two index questions from \cite{Merry2010a},
which are both related to the \textbf{Rabinowitz Floer homology} $RFH_{*}(\Sigma,T^{*}M)$
of the hypersurface $\Sigma$ (although both questions are valid in
situations where the Rabinowitz Floer homology has not yet been defined,
see Section \ref{sub:Dropping-the-assumption}).

\subsection{The Morse index of the free time action functional}

$\ $\vspace{6pt}

The first circle of ideas we study here is only indirectly related
to Rabinowitz Floer homology \cite{AbbondandoloSchwarz2009,Merry2010a}.
We begin with some notation. Denote by $\Lambda M$ the completion
of $C^{\infty}(S^{1},M)$ with respect to the Sobolev $W^{1,2}$-norm.
The set $\Lambda M$ carries the structure of a Hilbert manifold.
Given a free homotopy class $\alpha\in[S^{1},M]$, let $\Lambda_{\alpha}M$
denote the component of $\Lambda M$ consisting of loops $q:S^{1}\rightarrow M$
such that $[q]=\alpha$, and fix reference loops $q_{\alpha}\in\Lambda_{\alpha}M$.
Given $q\in\Lambda_{\alpha}M$, let $\bar{q}\in C^{0}(S^{1}\times[0,1],M)\cap W^{1,2}(S^{1}\times[0,1],M)$
denote any map such that $\bar{q}(t,0)=q(t)$ and $\bar{q}(t,1)=q_{\alpha}(t)$.
Then it is proved in \cite[p194]{Merry2010} that since we assume
that $\sigma$ is weakly exact and that $\widetilde{\sigma}$ admits
a bounded primitive, the value of $\int_{S^{1}\times[0,1]}\bar{q}^{*}\sigma$
is independent of the choice of $\bar{q}$.

This allows us to define the \textbf{free time action functional}
$\mathcal{S}_{k}$ on the product manifold $\Lambda M\times\mathbb{R}^{+}$:\[
\mathcal{S}_{k}:\Lambda M\times\mathbb{R}^{+}\rightarrow\mathbb{R};\]
 \[
\mathcal{S}_{k}(q,T):=T\int_{S^{1}}\left(L\left(q,\frac{\dot{q}}{T}\right)+k\right)dt+\int_{S^{1}\times[0,1]}\bar{q}^{*}\sigma.\]
 Denote by $\mbox{Crit}(\mathcal{S}_{k})$ the set of critical points
of $\mathcal{S}_{k}$. A pair $(q,T)$ is a critical point of $\mathcal{S}_{k}$
if and only if the curve $\gamma:[0,T]\rightarrow M$ defined by $\gamma(t):=q(t/T)$
is the projection to $M$ of a closed orbit of $\phi_{t}$ contained
in $\Sigma$ \cite[Corollary 2.3]{Merry2010}.

We can also fix the period of the free time action functional, thus
giving the \textbf{fixed period action functional}:\[
\mathcal{S}_{k}^{T}:\Lambda M\rightarrow\mathbb{R};\]
 \[
\mathcal{S}_{k}^{T}(q):=\mathcal{S}_{k}(q,T).\]
 Let $\left\langle \left\langle \cdot,\cdot\right\rangle \right\rangle $
denote the $W^{1,2}$ metric on $\Lambda M$ defined by \[
\left\langle \left\langle \zeta,\zeta'\right\rangle \right\rangle :=\int_{0}^{1}\left\langle \zeta,\zeta'\right\rangle +\left\langle \nabla_{t}\zeta,\nabla_{t}\zeta'\right\rangle dt,\]
where $\nabla$ is the Levi-Civita connection of $(M,g)$. We will
use the same notation for the product metric on $\Lambda M\times\mathbb{R}^{+}$:\[
\left\langle \left\langle (\zeta,b),(\zeta',b')\right\rangle \right\rangle :=\left\langle \left\langle \zeta,\zeta'\right\rangle \right\rangle +bb'.\]
 We denote by $\nabla\mathcal{S}_{k}$ and $\nabla\mathcal{S}_{k}^{T}$
the gradient of $\mathcal{S}_{k}$ and $\mathcal{S}_{k}^{T}$ with
respect to these metrics.

The \textbf{Morse index}\emph{ }$i(q,T)$ of a critical point $(q,T)\in\mbox{Crit}(\mathcal{S}_{k})$
is the maximal dimension of a subspace $W\subseteq W^{1,2}(S^{1},q^{*}TM)\times\mathbb{R}$
on which $d_{(q,T)}^{2}\mathcal{S}_{k}(\cdot,\cdot)$ is negative
definite. It is well known that for the Tonelli Lagrangians $L$ we
are working with the Morse index $i(q,T)$ is always finite. Similarly
let $i_{T}(q)$ denote the Morse index\textbf{ }of a critical point
$q\in\mbox{Crit}(\mathcal{S}_{k}^{T})$, that is, the dimension of
a maximal subspace of $W^{1,2}(S^{1},q^{*}TM)$ on which $d_{q}^{2}\mathcal{S}_{k}^{T}(\cdot,\cdot)$
is negative definite.

Note that \[
d_{q}\mathcal{S}_{k}^{T}(\zeta)=d_{(q,T)}\mathcal{S}_{k}(\zeta,0).\]
 Thus if $(q,T)\in\mbox{Crit}(\mathcal{S}_{k})$ then $q\in\mbox{Crit}(\mathcal{S}_{k}^{T})$.
It is therefore a natural question to ask how the two Morse indices
$i(q,T)$ and $i_{T}(q)$ are related for $(q,T)\in\mbox{Crit}(\mathcal{S}_{k})$.
It is clear that $0\leq i(q,T)-i_{T}(q)\leq1$, but the precise relationship
is somewhat more complicated, as we now explain.

\begin{defn}We say that a critical point $(q,T)\in\mbox{Crit}(\mathcal{S}_{k})$
admits an \textbf{orbit cylinder }if there exists $\varepsilon>0$
and a smooth (in $s$) family $\{(q_{k+s},T(k+s)\}_{s\in(-\varepsilon,\varepsilon)}$
of critical points of $\mathcal{S}_{k+s}$: \[
(q_{k+s},T(k+s))\in\mbox{Crit}(\mathcal{S}_{k+s}),\]
 where $(q_{k},T(k))=(q,T)$. We say that the orbit cylinder is \textbf{non-degenerate
}if $T'(k)\ne0$. In this case we define the \textbf{correction term}
\begin{equation}
\chi(q,T):=\mbox{sign}(-T'(k))\in\{-1,1\}.\label{eq:chi q T}\end{equation}

A priori, the correction term $\chi(q,T)$ depends on the choice of
orbit cylinder. However one consequence of Theorem \ref{thm:lag index-1}
below is that this is not the case.\end{defn}

This condition is the analogue of Assumption \ref{ass:nondegneracy}
in the Lagrangian setting. A sufficient condition for a critical point
$(q,T)$ to admit a non-degenerate orbit cylinder is that the corresponding
periodic orbit $y(t):=x(t/T)$ of $X_{H}$ (see Lemma \ref{lem:critical points}
below) is strongly non-degenerate. In this case the orbit cylinder
is actually \textbf{unique} (cf. the discussion surrounding \eqref{eq:matrix}
below).\newline 

Anyway, the precise relationship between the two Morse indices is
given by the following result.

\begin{thm} \label{thm:lag index-1}Let $(q,T)\in\mbox{\emph{Crit}}(\mathcal{S}_{k})$
denote a critical point admitting a non-degenerate orbit cylinder.
Then \[
i(q,T)=i_{T}(q)+\frac{1}{2}-\frac{1}{2}\chi(q,T).\]

\end{thm}

\subsection{Computing the virtual dimension of the moduli spaces of Rabinowitz
Floer homology}

$\ $\vspace{6pt}

The second functional we study is more directly related to Rabinowitz
Floer homology. Denote by $\Lambda T^{*}M$ the completion of $C^{\infty}(S^{1},T^{*}M)$
with respect to the Sobolev $W^{1,2}$-norm. The \textbf{Rabinowitz
action functional $\mathcal{A}_{k}$} is defined on the product manifold
$\Lambda T^{*}M\times\mathbb{R}$:

\[
\mathcal{A}_{k}:\Lambda T^{*}M\times\mathbb{R}\rightarrow\mathbb{R};\]
 \[
\mathcal{A}_{k}(x,\eta):=\int_{S^{1}}x^{*}\lambda_{0}+\int_{S^{1}\times[0,1]}\bar{q}^{*}\sigma-\eta\int_{S^{1}}(H(x)-k)dt,\]
 where $q:=\pi\circ x$, and $\bar{q}:S^{1}\times[0,1]\rightarrow M$
is defined as before. The critical points of $\mathcal{A}_{k}$ are
easily seen to satisfy:\[
\dot{x}=\eta X_{H}(x);\]
 \[
\int_{S^{1}}(H(x)-k)dt=0.\]
 Since $H$ is invariant under its Hamiltonian flow, the second equation
implies\[
x(S^{1})\subseteq\Sigma.\]
 Thus if $\mbox{Crit}(\mathcal{A}_{k})$ denotes the set of critical
points of $\mathcal{A}_{k}$, we can characterize $\mbox{Crit}(\mathcal{A}_{k})$
by\begin{eqnarray}
\mbox{Crit}(\mathcal{A}_{k}) & = & \left\{ (x,\eta)\in\Lambda T^{*}M\times\mathbb{R}\,:\, x\in C^{\infty}(S^{1},T^{*}M)\right.\label{eq:critA}\\
 &  & \left.\ \ \ \dot{x}=\eta X_{H}(x),\ x(S^{1})\subseteq\Sigma\right\} .\nonumber 
\end{eqnarray}
 The critical points of $\mathcal{A}_{k}$ with $\eta>0$ correspond
bijectively to the orbits of $\phi_{t}$ contained in $\Sigma$: if
$(x,\eta)\in\mbox{Crit}(\mathcal{A}_{k})$ with $\eta>0$ and $y:\mathbb{R}/\eta\mathbb{Z}\rightarrow\Sigma$
is defined by $y(t):=x(t/\eta)$ then $y(t)=\phi_{t}(y(0))$. Thus
the critical points of $\mathcal{A}_{k}$ are intimately related to
those of $\mathcal{S}_{k}$. The following lemma \cite[Lemma 4.1]{Merry2010a}
makes this precise. Denote by $\mathfrak{L}:TM\rightarrow T^{*}M$
the \textbf{Legendre transform }of $L$, defined by \begin{equation}
\mathfrak{L}(q,v):=\left(q,\frac{\partial L}{\partial v}(q,v)\right).\label{eq:legendre transform}\end{equation}

\begin{lem} \label{lem:critical points}Given $(q,T)\in\mbox{\emph{Crit}}(\mathcal{S}_{k})$,
define \[
x(t):=(q(t),\mathfrak{L}(q(t),\dot{q}(t))),\ \ \ x^{-}(t):=x(-t).\]
 Then\[
\mbox{\emph{Crit}}(\mathcal{A}_{k})=\left\{ (x,T),(x^{-},-T)\,:\,(q,T)\in\mbox{\emph{Crit}}(\mathcal{S}_{k})\right\} \cup\left\{ (x,0)\in\Sigma\times\{0\}\right\} .\]

\end{lem} 

We now explain precisely our standing non-degeneracy assumption. 

\begin{defn}\label{orbit cylinder}Given an orbit $y:\mathbb{R}/T\mathbb{Z}\rightarrow\Sigma$
of $X_{H}$, we say that $y$ admits an \textbf{orbit cylinder }if
there exists $\varepsilon>0$ together with a smooth (in $s$) family
$\mathcal{O}=(y_{k+s})_{s\in(-\varepsilon,\varepsilon)}$ of orbits
of $X_{H}$ \[
y_{k+s}:\mathbb{R}/T(k+s)\mathbb{Z}\rightarrow T^{*}M;\]
\[
H(y_{k+s})\equiv k+s,\]
with $y_{k}=y$. We say that the orbit cylinder is \textbf{non-degenerate
}if $T'(k)\ne0$. We define the \textbf{correction term }associated
to an orbit admitting a non-degenerate orbit cylinder by \[
\chi(y):=\mbox{sign}(-T'(k)).\]
A priori, the correction term $\chi(y)$ depends on the choice of
the orbit cylinder. However it follows from Theorem \ref{thm:lag index-1}
and Lemma \ref{lem:critical points} that this is actually not the
case.\end{defn}

A sufficient condition for an orbit $y$ to admit an orbit cylinder
is that $y$ is \textbf{weakly non-degenerate}, by which we mean that
$y$ has exactly two \textbf{Floquet multipliers }equal to one (see
for instance \cite[Proposition 4.2]{HoferZehnder1994}). In this case
the orbit cylinder $\mathcal{O}$ is \textbf{unique}.\textbf{ }If
we assume in addition that $y$ is \textbf{strongly non-degenerate
}(i.e. \eqref{eq:snd} holds) then the orbit cylinder $\mathcal{O}$
is non-degenerate. Indeed, let $N$ denote a hypersurface inside of
$\Sigma$ which is transverse to $y(\mathbb{R}/T\mathbb{Z})$ at the
point $y(0)$ with $T_{y(0)}N$ equal to the symplectic orthogonal
$(T_{y(0)}\mathcal{O})^{\perp}$ of the tangent space to the orbit
cylinder at $y(0)$. Let $P_{y}:\mathcal{U}\rightarrow\mathcal{V}$
denote the associated \textbf{Poincar\textbf{\'e} map}, where $\mathcal{U}$
and $\mathcal{V}$ are neighborhoods of $y(0)$. $P$ is a diffeomorphism
that fixes $y(0)$. Then there exists a unique symplectic splitting
of $T_{y(0)}T^{*}M$ such that $d_{y(0)}\phi_{T}$ is given by\begin{equation}
d_{y(0)}\phi_{T}=\left(\begin{array}{ccc}
1 & -T'(k) & 0\\
0 & 1 & 0\\
0 & 0 & \begin{array}{ccc}
\\
 & d_{y(0)}P_{y}\\
\\
\end{array}
\end{array}\right).\label{eq:matrix}\end{equation}
 Here $\mathbb{1}-d_{y(0)}P_{z}$ is invertible. The assumption that
$\nu(y)=1$ therefore implies that $T'(k)\ne0$. 

\begin{rem} \label{rem:A-remark-on terminiology}A perhaps more standard
name for {}``strongly non-degenerate'' is \textbf{transversally
non-degenerate}. Unfortunately\textbf{ }it seems that the term {}``transversally
non-degenerate''\textbf{ }has been used in the literature to mean
both {}``strongly non-degenerate'' and {}``weakly non degenerate''.
We prefer to clearly differentiate the two conditions, as our results
are only valid under the stronger one. \end{rem} 

Now suppose $(x,\eta)\in\mbox{Crit}(\mathcal{A}_{k})$ with $\eta\ne0$.
Let $y:\mathbb{R}/\left|\eta\right|\mathbb{Z}\rightarrow\Sigma$ be
defined by $y(t):=x(t/\left|\eta\right|)$. By an abuse of language,
we will say that $(x,\eta)$ admits a non-degenerate orbit cylinder
if same is true for the orbit $y$. In this case we may define the
correction term \[
\chi(x,\eta):=\mbox{sign}(\eta)\chi(y).\]
 Note that if $(q,T)\in\mbox{Crit}(\mathcal{S}_{k})$ admits a non-degenerate
orbit cylinder and $x,x^{-}\in\Lambda T^{*}M$ are defined as in Lemma
\ref{lem:critical points} then both the critical points $(x,T)$
and $(x^{-},-T)$ of $\mathcal{A}_{k}$ admit non-degenerate orbit
cylinders, and

\begin{equation}
\chi(q,T)=\chi(x,T)=-\chi(x^{-},-T).\label{eq:equality of correction terms}\end{equation}

In the next definition we restrict to strongly non-degenerate critical
points.

\begin{defn}\label{defn:mu} Given a critical point $(x,\eta)\in\mbox{Crit}(\mathcal{A}_{k})$
admitting a non-degenerate orbit cylinder, we define an index $\mu_{\textrm{Rab}}(x,\eta)\in\mathbb{Z}$
by \[
\mu_{\textrm{Rab}}(x,\eta):=\mu_{\textrm{CZ}}(y)-\frac{1}{2}\chi(x,\eta)\]
 Here $y:\mathbb{R}/\left|\eta\right|\mathbb{Z}\rightarrow\Sigma$
is defined as before by $y(t):=x(t/\left|\eta\right|)$, and $\mu_{\textrm{CZ}}(y)\in\left(\frac{1}{2}\mathbb{Z}\right)\backslash\mathbb{Z}$
denotes the \textbf{Conley-Zehnder index}\emph{ }of $y$. See \cite{RobbinSalamon1993}
for the definition of the Conley-Zehnder index in the case where there
exist Floquet multipliers of $y$ that are equal to $1$ (note however
that our sign conventions match those of \cite{AbbondandoloPortaluriSchwarz2008}
not \cite{RobbinSalamon1993}).\end{defn} 

In fact, the index $\mu_{\textrm{Rab}}(x,\eta)$ can be identified
with a suitably defined \textbf{transverse Conley-Zehnder} index,
and thus can be defined even when the orbit $(x,\eta)$ does not admit
a non-degenerate orbit cylinder. See Subsection \ref{subsection:mu}.
In fact, this identification proves the following result:

\begin{prop}\label{no depend}The index $\mu_{\textrm{\emph{Rab}}}$
only depends on the hypersurface $\Sigma=H^{-1}(k)$ and not on the
actual Hamiltonian defining it. \end{prop}

\textbf{Duistermaat's Morse index theorem} \cite{Duistermaat1976}
says that if $(q,T)\in\mbox{Crit}(\mathcal{S}_{k})$ is strongly non-degenerate
and $y(t):=x(t/T)$ where $x\in\Lambda T^{*}M$ is defined as in Lemma
\ref{lem:critical points} then \[
i_{T}(q)=\mu_{\textrm{CZ}}(y)-\frac{1}{2}.\]
 This is proved in the twisted case for mechanical Hamiltonians in
\cite[Appendix A]{Merry2010a}. The proof there goes through for an
arbitrary Tonelli Hamiltonian; alternatively one could use the argument
of \cite[Lemma 3.12.5]{Merry2010b}, which works directly for any
Tonelli Hamiltonian. Combining Theorem \ref{thm:lag index-1} and
\eqref{eq:equality of correction terms} one sees that if $(q,T)\in\mbox{Crit}(\mathcal{S}_{k})$
is a strongly non-degenerate critical point of $\mathcal{S}_{k}$
and $x,x^{-}\in\Lambda T^{*}M$ are defined as in Lemma \ref{lem:critical points}
then \begin{equation}
i(q,T)=\mu_{\textrm{Rab}}(x^{+},T)=-\mu_{\textrm{Rab}}(x^{-},-T)\label{eq:index relation}\end{equation}
 (see \cite[Corollary 4.5]{Merry2010a}). Moreover as a corollary
of Proposition \ref{no depend} we deduce:

\begin{cor}The Morse index $i(q,T)$ depends only on the hypersurface
$\Sigma$ and not on the Lagrangian $L$.\end{cor}

This is interesting, as the same result is \textbf{not }true for the
fixed period Morse index $i_{T}(q)$.\newline 

Suppose now that $J$ is an $\omega$-compatible almost complex structure
on $T^{*}M$. By this we mean that $\omega(J\cdot,\cdot)$ defines
a Riemannian metric on $T^{*}M$. Using $J$ we can build an $L^{2}$-metric
$\left\langle \left\langle \cdot,\cdot\right\rangle \right\rangle _{J}$
on $\Lambda T^{*}M$ via\[
\left\langle \left\langle \xi,\xi'\right\rangle \right\rangle _{J}:=\int_{S^{1}}\omega(J\xi,\xi')dt.\]
 We will use the same notation for the product metric on $\Lambda T^{*}M\times\mathbb{R}^{+}$:\[
\left\langle \left\langle (\xi,b),(\xi',b')\right\rangle \right\rangle _{J}:=\left\langle \left\langle \xi,\xi'\right\rangle \right\rangle _{J}+bb'.\]
 Denote by $\nabla\mathcal{A}_{k}$ the gradient of $\mathcal{A}_{k}$
with respect to $\left\langle \left\langle \cdot,\cdot\right\rangle \right\rangle _{J}$.
Thus \[
\nabla\mathcal{A}_{k}(x,\eta)=\left(J(x)(\partial_{t}x-\eta X_{H}(x)),-\int_{S^{1}}(H(x)-k)dt\right).\]
 Let us denote by $\bar{\partial}_{\mathcal{A}_{k}}$ the \textbf{Rabinowitz
Floer operator }defined by\[
\bar{\partial}_{\mathcal{A}_{k}}:C^{\infty}(\mathbb{R}\times S^{1},T^{*}M)\times C^{\infty}(\mathbb{R},\mathbb{R})\rightarrow C^{\infty}(\mathbb{R}\times S^{1},TT^{*}M)\times C^{\infty}(\mathbb{R},\mathbb{R})\]
 \[
\bar{\partial}_{\mathcal{A}_{k}}(u):=\partial_{s}u+\nabla\mathcal{A}_{k}(u(s,\cdot)),\ \ \ u=(x,\eta).\]
 The operator $\bar{\partial}_{\mathcal{A}_{k}}$ extends to define
a section of a certain Banach bundle $\mathcal{E}$ over a certain
Banach manifold $\mathcal{M}\subseteq C^{0}(\mathbb{R}\times S^{1},T^{*}M\times\mathbb{R})\times C^{0}(\mathbb{R},\mathbb{R})$.
Roughly speaking, $\mathcal{M}$ consists of those maps $u=(x,\eta)$
that are of class $W_{\textrm{loc}}^{1,r}$ for some $r>2$, and satisfy
a certain prescribed behavior at infinity. Any $u\in\mathcal{M}$
with $\bar{\partial}_{\mathcal{A}_{k}}(u)=0$ is necessarily of class
$C^{\infty}$. Given two critical points $v_{\pm}:=(x_{\pm},\eta_{\pm})$
of $\mbox{Crit}(\mathcal{A}_{k})$ let us denote by $\mathcal{M}(v_{-},v_{+})\subseteq\mathcal{M}$
the set of zeros $u$ of $\bar{\partial}_{\mathcal{A}_{k}}$ that
submit to the asymptotic conditions\[
\lim_{s\rightarrow\pm\infty}u(s,\cdot)=v_{\pm}.\]

Note that in the definition of $\mathcal{M}(v_{-},v_{+})$ we are
\textbf{not} dividing through by the $\mathbb{R}$-action given by
translating along the gradient flow lines.\\
 Recall that a space $\mathcal{N}$ is said to have \textbf{virtual
dimension }$m\in\mathbb{Z}$ if $\mathcal{N}$ can be seen as the
set of zeros of a smooth section of some Banach bundle, whose linearization
is Fredholm of index $m$. If such a section is transverse to the
zero-section then the implicit function theorem implies that $\mathcal{N}=\emptyset$
or $m\geq0$ and $\mathcal{N}$ carries the structure of a smooth
$m$-dimensional manifold. In our second result we extend Proposition
4.1 of \cite{CieliebakFrauenfelder2009} to the situation at hand
and compute the virtual dimension of $\mathcal{M}(v_{-},v_{+})$ in
the case where $v_{\pm}$ are strongly non-degenerate critical points.
\begin{thm} \label{thm:rab index-1}Let $v_{\pm}=(x_{\pm},\eta_{\pm})\in\mbox{\emph{Crit}}(\mathcal{A}_{k})$
with $\eta_{\pm}\ne0$ denote critical points of $\mathcal{A}_{k}$
that admit non-degenerate orbit cylinders. Then the space $\mathcal{M}(v_{-},v_{+})$
has virtual dimension \[
\mbox{\emph{virdim }}\mathcal{M}(v_{-},v_{+})=\mu_{\textrm{\emph{Rab}}}(v_{-})-\mu_{\textrm{\emph{Rab}}}(v_{+})-1.\]

\end{thm} We emphasize that the {}``-1'' in the formula above comes
from the fact that we are in a Morse-Bott situation, where the critical
manifolds have dimension 1. The method we use in order to compute
the virtual dimension differs in some respect to that of \cite[Proposition 4.1]{CieliebakFrauenfelder2009}.
In both cases the key part of the calculation is the computation of
a {}``correction term'' between the spectral flow of the Rabinowitz
action functional and the spectral flow of the classical Hamiltonian
action functional. The difference is that in \cite[Section 4]{CieliebakFrauenfelder2009}
the authors first perturb the hypersurface in such a way that the
{}``correction term'' is easy to compute \cite[Lemma C.6]{CieliebakFrauenfelder2009}.
In contrast, we work directly with the Hamiltonian $H$. However as
a consequence we need to assume that the given orbits are strongly
non-degenerate, whereas in \cite[Section 4]{CieliebakFrauenfelder2009}
they require only that the Rabinowitz action functional is Morse-Bott
on the component of the critical set containing the given orbits.

We remark that we are not assuming at any point that our hypersurfaces
are of contact type or stable. We also note that most likely our orientability
assumption is not necessary and the arguments of this paper can be extended using the trivializations from \cite{Weber2002}.

\subsection{Relationship with Rabinowitz Floer homology and Ma\~n\'e's critical
values}

$\ $\vspace{6pt}

So far we have not placed any restriction on the value $k\in\mathbb{R}$,
apart from asking for $k$ to be a regular value of $H$. In fact,
to be able to define and compute the Rabinowitz Floer homology of
$\Sigma$, or to be able to do Morse homology with $\mathcal{S}_{k}$,
we need to make an additional assumption on $k$. One option is to
assume that $k$ is sufficiently large. More precisely, there are
two particular {}``critical values'' $c$ and $c_{0}$ of $k$,
known as the\textbf{ Ma\textbf{\~n\'e} critical values}%
\footnote{$c$ is sometimes also denoted as $c_u$. %
}. They are such that the dynamics of the hypersurface $H^{-1}(k)$
differs dramatically depending on the relation of $k$ to these numbers.
They satisfy $c<\infty$ if and only if $\widetilde{\sigma}$ admits
a bounded primitive, and $c_{0}<\infty$ if and only if $\sigma$
is actually exact, and are defined as follows: let $\widetilde{H}:T^{*}\widetilde{M}\rightarrow\mathbb{R}$
denote the lift of $H$ to $T^{*}\widetilde{M}$, and set

\begin{equation}
c:=\inf_{\theta}\sup_{q\in\widetilde{M}}\widetilde{H}(q,\theta_{q}),\label{eq:mcv1}\end{equation}
where the infimum is taken over all $1$-forms $\theta$ on $\widetilde{M}$
with $d\theta=\widetilde{\sigma}$. If $\sigma$ is not exact define
$c_{0}:=\infty$. Otherwise define \begin{equation}
c_{0}:=\inf_{\theta}\sup_{q\in M}H(q,\theta_{q}),\label{eq:mcv2}\end{equation}
that is, the same definition only working directly on $T^{*}M$ rather
than lifting to $T^{*}\widetilde{M}$. See for instance \cite[Section 5]{CieliebakFrauenfelderPaternain2010}
for more information. 

In \cite{Merry2010a} the first author computed%
\footnote{Technically speaking, in \cite{Merry2010a} only mechanical Hamiltonians,
rather than Tonelli Hamiltonians are considered, although the proofs
go through in this more general setting.%
} the Rabinowitz Floer homology of $\Sigma$ for $k>c$, which builds
on the earlier construction of Abbondandolo and Schwarz \cite{AbbondandoloSchwarz2009}
which computed the Rabinowitz Floer homology of $\Sigma$ for $k>c_{0}$.

To get started one first assumes that all periodic orbits of $X_{H}$
contained in $\Sigma$ are strongly non-degenerate. Then the key idea
is to define chain maps between the Rabinowitz Floer chain complex,
which is generated by the critical points of $\mathcal{A}_{k}$, and
the Morse (co)chain complex generated by the critical points of $\mathcal{S}_{k}$.
The motivation for such a construction is provided by Lemma \ref{lem:critical points}
and the relationship \eqref{eq:index relation}.\\

A further remark to make is that in the setting of Rabinowitz Floer
homology one is free to use any Hamiltonian $F:T^{*}M\rightarrow\mathbb{R}$
such that $F^{-1}(k)=\Sigma$ and $X_{F}|_{\Sigma}$ coincides with
the Reeb vector field of the hypersurface. If $\Sigma$ is of \textbf{restricted
contact type }then one can choose such a Hamiltonian $F$ such that
$F$ near $\Sigma$ is radial and homogeneous of degree 2 in the radial
direction provided by the symplectization of the hypersurface. In
this case it is not hard to see that the correction term of a strongly
non-degenerate orbit $(x,\eta)$ is always given by the sign of $\eta$.
If $k>c_{0}$ then one can always find such a Hamiltonian (see \cite[Section 10]{AbbondandoloSchwarz2009});
this explains why in \cite{AbbondandoloSchwarz2009} the correction
terms do not appear explicitly. If we just ask that $k>c$ then it
is convenient for several reasons to keep the physical Hamiltonian
$H$; thus in \cite{Merry2010a} the correction term does appear.
One compelling reason to keep $H$ unchanged is that for $k>c$, the
hypersurface is now only virtually contact and not necessarily contact,
see \cite{CieliebakFrauenfelderPaternain2010}.

A natural question to ask therefore is whether if for $k>c$ the correction
terms $\chi(q,T)$ and $\chi(x,\eta)$ (for $\eta>0$) are always
$+1$. In the last section of this note we answer this question in
the negative. More precisely, we construct an example of a mechanical
Hamiltonian $H$ on $(T^{*}S^{2},\omega_{\sigma})$ (for a suitable
choice of magnetic form $\sigma$) for which there exist a strongly
non-degenerate orbit $y$ of $X_{H}$ with energy $k>c$ (in this
case $c=c_{0}$) for which the correction term $\chi(y)$ is equal
to $-1$. (It is easily seen that the example can be arranged so that
it also carries orbits for which the correction term is $+1$.) In
particular, for this example Theorem \ref{thm:lag index-1} implies
that $i(q,T)=i_{T}(q)+1$.

\subsection{\label{sub:Dropping-the-assumption}Dropping the assumption that
$\sigma$ is weakly exact and admits a bounded primitive}

$\ $\vspace{6pt}

We conclude this introduction with the following remark. None of the
results in this paper actually use the fact that $\sigma$ is weakly
exact and admits a bounded primitive, at least if one is prepared
to work with 1-forms rather than functionals. More precisely, suppose
we consider the $1$-form $\mathfrak{a}_{k}\in\Omega^{1}(\Lambda T^{*}M\times\mathbb{R})$
defined by \[
\mathfrak{a}_{k}(x,\eta)(\xi,b):=\int_{S^{1}}\omega(\xi,\dot{x}-\eta X_{H}(x))dt-b\int_{S^{1}}H(x)dt\ \ \ \mbox{for}\ \ \ (\xi,b)\in T_{(x,\eta)}(\Lambda T^{*}M\times\mathbb{R}).\]
 Under the assumption that $\sigma$ is closed, $\mathfrak{a}_{k}$
is a well defined closed\textbf{ }1-form on $\Lambda T^{*}M\times\mathbb{R}$.
Of course, if $\sigma$ is weakly exact and admits a bounded primitive
then $\mathfrak{a}_{k}$ is exact, with \[
\mathfrak{a}_{k}(x,\eta)=d_{(x,\eta)}\mathcal{A}_{k}.\]
 The critical points of $\mathfrak{a}_{k}$, that is, the points $(x,\eta)\in\Lambda T^{*}M\times\mathbb{R}$
with $a_{k}(x,\eta)=0$, are still given by \eqref{eq:critA}.

In a similar vein, if we only assume that $\sigma$ is closed then
we can define a closed 1-form $\mathfrak{s}_{k}$ on $\Lambda M\times\mathbb{R}^{+}$.
As with $\mathfrak{a}_{k}$, if we make the additional assumptions
that $\sigma$ is weakly exact and admits a bounded primitive, then
$\mathfrak{s}_{k}$ is exact and satisfies \[
\mathfrak{s}_{k}(q,T)=d_{(q,T)}\mathcal{S}_{k}.\]
 As with $\mathcal{S}_{k}$, the critical points of $\mathfrak{s}_{k}$
are the pairs $(q,T)$ such that if $\gamma:[0,T]\rightarrow M$ is
defined by $\gamma(t):=q(t/T)$ then $\gamma$ is the projection to
$M$ of a closed orbit of $\phi_{t}$ contained in $\Sigma$.

Both Theorem \ref{thm:lag index-1} and Theorem \ref{thm:rab index-1}
continue to make sense if we work only with the 1-forms $\mathfrak{s}_{k}$
and $\mathfrak{a}_{k}$ respectively, and all the proofs in this note
go through word for word in this more general setting. Hopefully this
observation will be useful in defining a more general Rabinowitz Floer
homology for twisted cotangent bundles where $\sigma$ is only assumed
closed. 

\begin{acknowledgement*} We are deeply grateful to the anonymous
referee for suggesting a considerable simplification to the proof
of Theorem \ref{thm:lag index-1} (cf. Remark \ref{referee}). We
would also like to thank Alberto Abbondandolo for several helpful
comments and suggestions on an earlier draft of this note and to Urs
Frauenfelder for his remark that led to Subsection \ref{subsection:mu}.
\end{acknowledgement*}

\section{Proofs}

We denote by $\overline{\mathbb{R}}$ the extended real line $\overline{\mathbb{R}}:=\mathbb{R}\cup\{\pm\infty\}$,
with the differentiable structure induced by the bijection $[-\pi/2,\pi/2]\rightarrow\overline{\mathbb{R}}$
given by \[
s\mapsto\begin{cases}
\tan s & s\in(-\pi/2,\pi/2)\\
\pm\infty & s=\pm\pi/2.
\end{cases}\]
 \textbf{All the sign conventions used in this paper match those of
\cite{AbbondandoloSchwarz2009,Merry2010a}}.

\subsection{\label{sub:The-proof-of}The proof of Theorem \ref{thm:lag index-1}}

$\ $\vspace{6pt}

Denote by $\nabla^{2}\mathcal{S}_{k}(q,T)$ and $\nabla^{2}\mathcal{S}_{k}^{T}(q)$
the $W^{1,2}$-Hessians of $\mathcal{S}_{k}$ and $\mathcal{S}_{k}^{T}$,
so that if $(q,T)\in\mbox{Crit}(\mathcal{S}_{k})$ then \[
d_{(q,T)}^{2}\mathcal{S}_{k}((\zeta,b),(\zeta',b'))=\left\langle \left\langle \nabla^{2}\mathcal{S}_{k}(q,T)(\zeta,b),(\zeta',b')\right\rangle \right\rangle ;\]
\[
d_{(q,T)}^{2}\mathcal{S}_{k}(\zeta,\zeta')=\left\langle \left\langle \nabla^{2}\mathcal{S}_{k}^{T}(q)(\zeta),\zeta'\right\rangle \right\rangle .\]
 The assertion of the theorem is local, so without loss of generality
we may suppose $M=\mathbb{R}^{n}$. Suppose $(q_{k+s},T(k+s))_{s\in(-\varepsilon,\varepsilon)}$
is an orbit cylinder about $(q_{k},T(k))$. Let \[
\zeta_{k}(t):=\frac{\partial}{\partial s}q_{k+s}(t).\]
Since $(q_{k+s},T(k+s))$ is a critical point of $\mathcal{S}_{k+s}$,
for any $(\zeta,b)\in W^{1,2}(S^{1},\mathbb{R}^{n})\times\mathbb{R}$
we have \[
d_{(q_{k+s},T(k+s))}\mathcal{S}_{k+s}(\zeta,b)=0.\]
We differentiate this equation with respect to $s$ and evaluate at
$s=0$ to obtain\begin{align*}
0 & =\frac{\partial}{\partial s}\Bigl|_{s=0}d_{(q_{k+s},T(k+s))}\mathcal{S}_{k+s}(\zeta,b)\\
 & =\frac{\partial}{\partial s}\Bigl|_{s=0}d_{(q_{k},T(k))}\mathcal{S}_{k+s}(\zeta,b)+d_{(q_{k},T(k))}^{2}\mathcal{S}_{k}((\zeta_{k},T'(k)),(\zeta,b)).
\end{align*}
In order to compute $\frac{\partial}{\partial s}\Bigl|_{s=0}d_{(q_{k},T(k))}\mathcal{S}_{k+s}(\zeta,b)$,
choose a variation $(q_{k,r},T(k,r))$ for $r\in(-\varepsilon,\varepsilon)$
such that $(q_{k,0},T(k,0))=(q_{k},T(k))$ and $\frac{\partial}{\partial r}\Bigl|_{r=0}(q_{k,r},T(k,r))=(\zeta,b)$.
Then \begin{align}
\frac{\partial}{\partial s}\Bigl|_{s=0}d_{(q_{k},T(k))}\mathcal{S}_{k+s}(\zeta,b) & =\frac{\partial}{\partial s}\Bigl|_{s=0}\frac{\partial}{\partial r}\Bigl|_{r=0}\mathcal{S}_{k+s}(q_{k,r},T(k,r))\label{eq:why it's b}\\
 & =\frac{\partial}{\partial r}\Bigl|_{r=0}\frac{\partial}{\partial s}\Bigl|_{s=0}\mathcal{S}_{k+s}(q_{k,r},T(k,r))\nonumber \\
 & =\frac{\partial}{\partial r}\Bigl|_{r=0}T(k,r)\nonumber \\
 & =b.\nonumber 
\end{align}
In other words, we have proved:\[
d_{(q_{k},T(k))}^{2}\mathcal{S}_{k}((\zeta_{k},T'(k)),\zeta,b))=-b.\]
Taking $b=0$ we see that \begin{equation}
d_{(q_{k},T(k))}^{2}\mathcal{S}_{k}((\zeta_{k},T'(k)),(\zeta,0))=0,\label{eq:ort1}\end{equation}
and taking $(\zeta,b)=(\zeta_{k},T'(k))$ we see that \[
d_{(q_{k},T(k))}^{2}\mathcal{S}_{k}((\zeta_{k},T'(k)),(\zeta_{k},T'(k)))=-T'(k).\]
Let us write $W^{1,2}(S^{1},\mathbb{R}^{n})\times\mathbb{R}=E^{+}\oplus\ker(\nabla^{2}\mathcal{S}_{k}(q_{k},T(k))\oplus E^{-}$,
where $E^{+(-)}$ is the positive (negative) eigenspace of $\nabla^{2}\mathcal{S}_{k}(q_{k},T(k))$.
Similarly write $W^{1,2}(S^{1},\mathbb{R}^{n})=E_{T(k)}^{+}\oplus\ker(\nabla^{2}\mathcal{S}_{k}^{T(k)}(q_{k}))\oplus E_{T(k)}^{-}$,
where $E_{T(k)}^{+(-)}$ is the positive (negative) eigenspace of
$\nabla^{2}\mathcal{S}_{k}^{T(k)}(q_{k})$ (such spectral decompositions
exist as $\nabla^{2}\mathcal{S}_{k}(q_{k},T(k))$ and $\nabla^{2}\mathcal{S}_{k}^{T(k)}(q_{k})$
are bounded self-adjoint operators). Clearly $E_{T(k)}^{\pm}\times\{0\}\subseteq E^{\pm}$.
Equation \eqref{eq:ort1} tells us that the 1-dimensional vector space
$W:=\mbox{span}_{\mathbb{R}}\left\langle (\zeta_{k},T'(k))\right\rangle $
is orthogonal to $W^{1,2}(S^{1},\mathbb{R}^{n})\times\{0\}$, and
thus we must have \[
E^{+}=E_{T(k)}^{+}\times\{0\},\ \ \ E^{-}=(E_{T(k)}^{-}\times\{0\})\oplus W\ \ \ \mbox{if }T'(k)>0,\]
or\[
E^{+}=(E_{T(k)}^{+}\times\{0\})\oplus W,\ \ \ E^{-}=E_{T(k)}^{-}\times\{0\}\ \ \ \mbox{if }T'(k)<0.\]
The theorem follows.

\subsection{\label{sub:The-spectral-flow}The spectral flow}

$\ $\vspace{6pt}

In this section we recall the definition of the \textbf{spectral flow},
following the exposition in \cite{RobbinSalamon1995}, and then recall
the statement of a theorem of Cieliebak and Frauenfelder \cite[Theorem C.5]{CieliebakFrauenfelder2009},
which we use to prove Theorem \ref{thm:rab index-1}.\textbf{ }Let
$W$ and $H$ denote a pair of separable real Hilbert spaces with
$W\subseteq H=H^{*}\subseteq W^{*}$ such that the inclusion $W\hookrightarrow H$
is compact with dense range. Let $\mathbf{L}(W,H)$ denote the set
of bounded linear operators, and let $\mathbf{S}(W,H)\subseteq\mathbf{L}(W,H)$\textbf{
}denote the subspace of self-adjoint operators (considered as unbounded
operators on $H$ with dense domain $W$).

Denote by $\mathbf{A}(W,H)$ the set of continuous (with respect to
the norm topologies) maps $A:\mathbb{R}\rightarrow\mathbf{S}(W,H)$
and such that the limits\[
A^{\pm}:=\lim_{s\rightarrow\pm\infty}A(s)\]
 exist and belong to $\mathbf{S}(W,H)$. Denote by $\mathbf{A}_{0}(W,H)\subseteq\mathbf{A}(W,H)$
the subset consisting of those elements $A\in\mathbf{A}(W,H)$ such
that the limit operators $A^{\pm}$ are bijective onto $W$. $ $The
\textbf{spectral flow }is a map \[
\mu_{\textrm{SF}}:\mathbf{A}_{0}(W,H)\rightarrow\mathbb{Z}\]
 which is characterized by the following four properties. 
\begin{itemize}
\item $\mu_{\textrm{SF}}$ is constant on the connected components of $\mathbf{A}_{0}(W,H)$. 
\item If $A$ is a constant map, $\mu_{\textrm{SF}}(A)=0$. 
\item $\mu_{\textrm{SF}}(A_{0}\oplus A_{1})=\mu_{\textrm{SF}}(A_{0})+\mu_{\textrm{SF}}(A_{1})$. 
\item If $W=H$ are finite dimensional then \[
\mu_{\textrm{SF}}(A)=\frac{1}{2}\mbox{sign}(A^{+})-\frac{1}{2}\mbox{sign}(A^{-}),\]
 where for a matrix $C$, $\mbox{sign}(C)$ denotes the number of
positive eigenvalues of $C$ minus the number of negative eigenvalues
of $C$. 
\end{itemize}
The spectral flow may be extended to a function $\mu_{\textrm{SF}}:\mathbf{A}(W,H)\rightarrow\mathbb{Z}$
as follows. Suppose\textbf{ }$A\in\mathbf{A}(W,H)$. Since the limit
operators $A^{\pm}$ are self-adjoint, their spectrums $\sigma(A^{\pm})$
are discrete and contained in $\mathbb{R}$. Hence \[
\lambda:=\inf\left\{ \left|\rho\right|\,:\,\rho\in(\sigma(A^{+})\cup\sigma(A^{-}))\backslash\{0\}\right\} >0.\]
 Thus if we choose $0<\delta<\lambda$ then the operators \[
A_{\delta}^{\pm}:=A^{\pm}\mp\delta\mathbb{1}\]
 are bijective. Let $\beta\in C^{\infty}(\mathbb{R},[-1,1])$ denote
a smooth cutoff function such that $\beta(s)=1$ for $s\geq1$ and
$\beta(s)=-1$ for $s\leq0$, and set \[
A_{\delta}(s):=A(s)-\delta\beta(s)\mathbb{1},\]
 where $\delta>0$ is as above. Then $A_{\delta}\in\mathbf{A}_{0}(W,H)$,
and hence the spectral flow $\mu_{\textrm{SF}}(A_{\delta})$ is well
defined. We may therefore define\[
\mu_{\textrm{SF}}(A):=\lim_{\delta\searrow0}\mu_{\textrm{SF}}(A_{\delta}).\]
 Note that this limit stabilizes for $\delta>0$ sufficiently small.

Suppose now we are given maps $A\in\mathbf{A}(W,H)$, $h\in\mathbf{A}(\mathbb{R},H)$
and $\tau\in\mathbf{A}(\mathbb{R},\mathbb{R})$. Define the map $A_{h,\tau}\in\mathbf{A}(W\oplus\mathbb{R},H\oplus\mathbb{R})$
by:\[
A_{h,\tau}(s)\left(\begin{array}{c}
v\\
b
\end{array}\right)=\left(\begin{array}{cc}
A(s) & h(s)\\
h^{*}(s) & \tau(s)
\end{array}\right)\left(\begin{array}{c}
v\\
b
\end{array}\right)=\left(\begin{array}{c}
A(s)(v)+bh(s)\\
\left\langle h(s),v\right\rangle _{H}+\tau(s)b
\end{array}\right).\]
 The spectral flow of the operator $A_{h,\tau}$ is then given by:\[
\mu_{\textrm{SF}}(A_{h,\tau}):=\lim_{\delta\searrow0}\mu_{\textrm{SF}}((A_{h,\tau})_{\delta}).\]
 If $\tau\equiv0$ we write $A_{h}$ instead of $A_{h,0}$.

Before stating the theorem of Cieliebak and Frauenfelder, we introduce
the following terminology. \begin{defn} \label{def:regular}Say that
the triple $(A,h,\tau)\in\mathbf{S}(W,H)\times H\times\mathbb{R}$
is \textbf{regular }if $h\in W\cap\mbox{range}(A)$ and if $v\in W$
is any vector such that $Av=h$, then the real number $\lambda_{A,h}:=\left\langle v,h\right\rangle _{H}$
is not equal to $\tau$. Note $h$ is orthogonal to $\ker(A)$ as
$A$ is self-adjoint. Thus $\lambda_{A,h}$ does not depend on the
choice of $v$ such that $Av=h$: if $v'$ is another choice then
$v-v'\in\ker(A)$ and hence $\left\langle v-v',h\right\rangle _{H}=0$
as $h\perp\ker(A)$. \end{defn} The next result is \cite[Theorem C.5]{CieliebakFrauenfelder2009}.
\begin{thm} \label{thm:CF}\textbf{\emph{(Cieliebak, Frauenfelder)}}

Let $A\in\mathbf{A}(W,H)$, $h\in\mathbf{A}(\mathbb{R},H)$ and $\tau\in\mathbf{A}(\mathbb{R},\mathbb{R})$.
Assume that the limit operators $(A^{\pm},h^{\pm},\tau^{\pm})$ are
regular triples. Then \[
\mu_{\textrm{\emph{SF}}}(A_{h,\tau})=\mu_{\textrm{\emph{SF}}}(A)+\frac{1}{2}\mbox{\emph{sign}}(\tau^{+}-\lambda_{A^{+},h^{+}})-\frac{1}{2}\mbox{\emph{sign}}(\tau^{-}-\lambda_{A^{-},h^{-}}).\]

\end{thm}

\begin{rem}\label{referee}In fact, we use Theorem \ref{thm:CF}
only in the case where $\tau\equiv0$ (this is also how the theorem
is stated in \cite{CieliebakFrauenfelder2009}, although the proof
goes through without change). In an earlier version of this paper
we proved Theorem \ref{thm:lag index-1} by applying Theorem \ref{thm:CF}
to the Hessian of $\mathcal{S}_{k}$, and arguing as in \cite[Corollary 6.41]{RobbinSalamon1995}.
The referee pointed out that Theorem \ref{thm:lag index-1} follows
from the simpler argument given in Subsection \ref{sub:The-proof-of}
above; we are very grateful to him or her for this observation.\end{rem}

\subsection{\label{sub:The-proof-of rab index}The proof of Theorem \ref{thm:rab index-1}}

$\ $\vspace{6pt}

In this section we prove Theorem \ref{thm:rab index-1}. Before getting
started, we introduce some more notation, Suppose $(x,\eta)\in\mbox{Crit}(\mathcal{A}_{k})$
with $\eta>0$. Let us denote by \[
\mathcal{A}_{k}^{\eta}:\Lambda T^{*}M\rightarrow\mathbb{R}\]
 the \textbf{classical action functional }of Hamiltonian mechanics:\[
\mathcal{A}_{k}^{\eta}(x):=\mathcal{A}_{k}(x,\eta).\]
 Let $\nabla\mathcal{A}_{k}^{\eta}$ denote the $L^{2}$-gradient
of $\mathcal{A}_{k}^{\eta}$ with respect to the metric $\left\langle \left\langle \cdot,\cdot\right\rangle \right\rangle _{J}$
on $\Lambda T^{*}M$, and as above denote by $\nabla^{2}\mathcal{A}_{k}(x,\eta)$
and $\nabla^{2}\mathcal{A}_{k}^{\eta}(x)$ the $L^{2}$-Hessians of
$\mathcal{A}_{k}$ and $\mathcal{A}_{k}^{\eta}$ at a critical point
$(x,\eta)$ of $\mathcal{A}_{k}$. One checks that $\nabla^{2}\mathcal{A}_{k}(x,\eta)$
and $\nabla^{2}\mathcal{A}_{k}^{\eta}(x)$ are related as follows:
\begin{align}
\nabla^{2}\mathcal{A}_{k}(x,\eta)(\xi,b) & =\left(\begin{array}{cc}
\nabla^{2}\mathcal{A}_{k}^{\eta}(x) & -\nabla H(x)\\
-\nabla H(x)^{*} & 0
\end{array}\right)\left(\begin{array}{c}
\xi\\
b
\end{array}\right)\label{eq:hessian}\\
 & =\left\langle \left\langle \nabla^{2}\mathcal{A}_{k}^{\eta}(x),\xi\right\rangle \right\rangle _{J}-b\left\langle \left\langle \nabla H(x),\xi\right\rangle \right\rangle _{J}.\nonumber 
\end{align}
Now fix two critical points $v_{\pm}=(x_{\pm},\eta_{\pm})\in\mbox{Crit}(\mathcal{A}_{k})$.
In order to show that $\mathcal{M}(v_{-},v_{+})$ has virtual dimension
$\mu_{\textrm{Rab}}(v_{-})-\mu_{\textrm{Rab}}(v_{+})-1$ we need to
show that the linearization $D_{u}$ of $\bar{\partial}_{\mathcal{A}_{k}}$
at any $u\in\mathcal{M}(v_{-},v_{+})$ is Fredholm, and compute its
index. For simplicity we will prove the theorem only in the case where
both $\eta_{-}$ and $\eta_{+}$ are positive. The case when $\eta_{\pm}$
have arbitrary sign is analogous.

Suppose $u=(x,\eta)\in\mathcal{M}(v_{-},v_{+})$. It is well known
that any such map $u$ extends to a map $u:\overline{\mathbb{R}}\rightarrow\Lambda T^{*}M\times\mathbb{R}$,
which we continue to denote by $u$. Fix $r\geq2$. The linearization
of $D_{u}$ of $\bar{\partial}_{\mathcal{A}_{k}}$ at $u$ is given
by: \[
D_{u}:W^{1,r}(\mathbb{R}\times S^{1},x^{*}TT^{*}M)\times W^{1,r}(\mathbb{R},\mathbb{R})\rightarrow L^{r}(\mathbb{R}\times S^{1},x^{*}TT^{*}M)\times L^{r}(\mathbb{R},\mathbb{R});\]
 \[
D_{u}\left(\begin{array}{c}
\xi\\
b
\end{array}\right)=\left(\begin{array}{c}
\nabla_{s}\xi+J(x)\nabla_{t}\xi+\nabla_{\xi}J(\dot{x})-\eta\nabla_{\xi}\nabla H(x)-b\nabla H(x)\\
\partial_{s}b-\int_{0}^{1}\left\langle \nabla H,\xi\right\rangle _{J}dt
\end{array}\right).\]
 In order to compute the index we express $D_{u}$ in a local trivialization.
Let $\varphi:\overline{\mathbb{R}}\times S^{1}\times\mathbb{R}^{2n}\rightarrow x^{*}TT^{*}M$
denote a symplectic trivialization of the pullback bundle over $\overline{\mathbb{R}}\times S^{1}$,
and denote by $\varphi_{\pm}:=\varphi|_{\{\pm\infty\}\times S^{1}}$
the induced symplectic trivializations of the pullback bundles $x_{\pm}^{*}TT^{*}M$.
Let $\widehat{\varphi}:\overline{\mathbb{R}}\times S^{1}\times\mathbb{R}^{2n}\times\mathbb{R}\rightarrow x^{*}TT^{*}M\times\eta^{*}T\mathbb{R}$
denote the product trivialization of $\varphi$ together with some
trivialization of the bundle $\eta^{*}T\mathbb{R}\rightarrow\overline{\mathbb{R}}$.
We now conjugate $D_{u}$ via $\widehat{\varphi}$ to obtain an operator
of the form\[
D_{A_{h}}:=\widehat{\varphi}^{-1}\circ D_{u}\circ\widehat{\varphi};\]
 \[
D_{A_{h}}:W^{1,r}(\mathbb{R}\times S^{1},\mathbb{R}^{2n})\times W^{1,r}(\mathbb{R},\mathbb{R})\rightarrow L^{r}(\mathbb{R}\times S^{1},\mathbb{R}^{2n})\times L^{r}(\mathbb{R},\mathbb{R});\]
 \[
D_{A_{h}}\left(\begin{array}{c}
v\\
b
\end{array}\right)(s)=(\partial_{s}+A_{h})\left(\begin{array}{c}
v\\
b
\end{array}\right)(s)=\left(\begin{array}{c}
\partial_{s}v(s)+A(s)(v)+bh(s)\\
\partial_{s}b(s)+\int_{0}^{1}\left\langle h(s),v\right\rangle dt
\end{array}\right).\]
 Here\begin{equation}
A(s,t):=J_{0}\partial_{t}+S(s,t),\label{eq:A}\end{equation}
 where\[
S(s,t)\in C^{\infty}(\mathbb{R}\times S^{1},\mathbf{L}(\mathbb{R}^{2n},\mathbb{R}^{2n}))\]
 is defined by\begin{equation}
S(s,t)=\varphi^{-1}\circ\left(\nabla_{s}\varphi+J(x)\nabla_{t}\varphi-\nabla_{\varphi}J(\dot{x})-\eta\nabla_{\varphi}\nabla H(x)\right)(s,t),\label{eq:S}\end{equation}
 and $J_{0}=\left(\begin{array}{cc}
0 & -\mathbb{1}\\
\mathbb{1} & 0
\end{array}\right)$ is the standard almost complex structure on $\mathbb{R}^{2n}$.

Secondly \[
h\in C^{\infty}(\mathbb{R}\times S^{1},\mathbb{R}^{2n})\]
 is the vector valued function \begin{equation}
h(s,t):=\varphi^{-1}\circ(-\nabla H(x))(s,t).\label{eq:B}\end{equation}
 As it stands the operator $D_{A_{h}}$ is \textbf{not }Fredholm.
This can be rectified however by defining $D_{A_{h}}$ instead on
a suitably \textbf{weighted Sobolev space}. This is explained in \cite[Lemma A.13]{Frauenfelder2004}.
Since the kernel and the cokernel of $D_{A_{h}}$ are spanned by smooth
elements, the index of $D_{A_{h}}$ does not depend on $r$, and hence
it suffices to consider the case $r=2$. In this case a well known
result of Robbin and Salamon \cite[Theorem 4.21]{RobbinSalamon1995}
tells us that the Fredholm index of $D_{A_{h}}$ is given by spectral
flow of $A_{h}$ :\[
\mbox{ind}(D_{A_{h}})=\mu_{\textrm{SF}}(A_{h}).\]
 We therefore need to compute this spectral flow. Note that the limit
operators are given by: \[
A^{\pm}=\varphi_{\pm}^{-1}\circ\nabla^{2}\mathcal{A}_{k}^{\eta_{\pm}}(x_{\pm})\circ\varphi_{\pm};\]
 \[
h^{\pm}:=\varphi_{\pm}^{-1}\circ(-\nabla H(x_{\pm})).\]
The computation in the next subsection will show that the triple $(A^{\pm},h^{\pm},0)$
is regular in the sense of Definition \ref{def:regular}, that is,
the correction terms $\lambda_{A^{\pm},h^{\pm}}$ are non-zero. Thus
by Theorem \ref{thm:CF} the spectral flow of $A_{h}$ is given by
\begin{equation}
\mu_{\textrm{SF}}(A_{h})=\mu_{\textrm{SF}}(A)-\frac{1}{2}\mbox{sign}(\lambda_{A^{+},h^{+}})+\frac{1}{2}\mbox{sign}(\lambda_{A^{-},h^{-}}).\label{eq:formula}\end{equation}
 A well known result of Salamon and Zehnder \cite{SalamonZehnder1992}
tells us that\[
\mu_{\textrm{SF}}(A)=\mu_{\textrm{CZ}}(y_{-})-\mu_{\textrm{CZ}}(y_{+})-1,\]
 where $y_{\pm}(t):=y_{\pm}(t/\eta_{\pm})$. Here the {}``$-1$''
in the formula comes from the fact that the $y_{\pm}$ have Floquet
multipliers equal to $1$ (see Theorem 4.2 and Lemma 4.4 in \cite{CieliebakFrauenfelder2009}).
In order to complete the proof of Theorem \ref{thm:rab index-1} we
therefore need to prove that \begin{equation}
\mbox{sign}(\lambda_{A^{\pm},h^{\pm}})=-\chi(x_{\pm},\eta_{\pm}).\label{eq:to prove}\end{equation}

\subsection{Computing $\lambda_{A^{\pm},h^{\pm}}$}

$\ $\vspace{6pt}

Let us start now with a critical point $(x_{k},\eta(k))\in\mbox{Crit}(\mathcal{A}_{k})$
that admits a non-degenerate orbit cylinder. Thus there exists $\varepsilon>0$
and a family $(x_{k+s},\eta(k+s))$ of critical points of $\mathcal{A}_{k+s}$
for $s\in(-\varepsilon,\varepsilon)$, with $\eta(k+s)>0$ for each
$s\in(-\varepsilon,\varepsilon)$ and $\eta'(k)\ne0$. Let \[
\xi_{k}(t):=\frac{\partial}{\partial s}\Bigl|_{s=0}x_{k+s}(t).\]
We will prove below that:\begin{equation}
\nabla^{2}\mathcal{A}_{k}^{\eta(k)}(x_{k})(\xi_{k})=\eta'(k)\nabla H(x_{k}).\label{eq:keyrab}\end{equation}
Equation \eqref{eq:to prove} readily follows from \eqref{eq:keyrab},
as we compute: \begin{align*}
\lambda_{\nabla^{2}\mathcal{A}_{k}^{\eta(k)}(x_{k}),-\nabla H(x_{k}),}: & =\int_{0}^{1}\left\langle -\frac{1}{\eta'(k)}\xi_{k},-\nabla H(x_{k})\right\rangle dt\\
 & =\frac{1}{\eta'(k)}\int_{0}^{1}d_{x_{k}}H\left(\frac{\partial}{\partial s}\Bigl|_{s=0}x_{k+s}(t)\right)dt\\
 & =\frac{1}{\eta'(k)}\int_{0}^{1}\frac{\partial}{\partial s}\Bigl|_{s=0}H(x_{k+s}(t))dt\\
 & =\frac{1}{\eta'(k)}\int_{0}^{1}\frac{\partial}{\partial s}\Bigl|_{s=0}\{k+s\}dt\\
 & =\frac{1}{\eta'(k)},
\end{align*}
and hence\[
\mbox{sign}(\lambda_{\nabla^{2}\mathcal{A}_{k}^{\eta(k)}(x_{k}),-\nabla H(x_{k})}):=\mbox{sign}(\eta'(k))=-\chi(x_{k},\eta(k)).\]
It remains therefore to check \eqref{eq:keyrab}. Set $\rho_{k}:=\nabla^{2}\mathcal{A}_{k}^{\eta(k)}(x_{k})(\xi_{k})-\eta'(k)\nabla H(x_{k})$;
we show $\rho_{k}=0$. We proceed as in Subsection \ref{sub:The-proof-of}.
The assertion is local, and hence we may assume $T^{*}M=\mathbb{R}^{2n}$.
For any $(\xi,b)\in W^{1,2}(S^{1},\mathbb{R}^{2n})\times\mathbb{R}$
we have\[
d_{(x_{k+s},\eta(k+s))}\mathcal{A}_{k+s}(\xi,b)=0,\]
and differentiating this equation with respect to $s$ and setting
$s=0$ we discover\begin{align*}
0 & =\frac{\partial}{\partial s}\Bigl|_{s=0}d_{(x_{k+s},\eta(k+s))}\mathcal{A}_{k+s}(\xi,b)\\
 & =\frac{\partial}{\partial s}\Bigl|_{s=0}d_{(x_{k},\eta(k))}\mathcal{A}_{k+s}(\xi,b)+d_{(x_{k},\eta(k))}^{2}\mathcal{A}_{k}((\xi_{k},\eta'(k)),(\xi,b))\\
 & =b+d_{(x_{k},\eta(k))}^{2}\mathcal{A}_{k}((\xi_{k},\eta'(k)),(\xi,b)),
\end{align*}
where the last line used a similar argument to \eqref{eq:why it's b}.

Taking $(\xi,b)=(\rho_{k},0)$ we see that\begin{align*}
0 & =\left\langle \left\langle \nabla^{2}\mathcal{A}_{k}(x_{k},\eta'(k))(\xi_{k},\eta'(k)),(\rho_{k},0)\right\rangle \right\rangle _{J}\\
 & \overset{(*)}{=}\left\langle \left\langle \left(\begin{array}{cc}
\nabla^{2}\mathcal{A}_{k}^{\eta(k)}(x_{k}) & -\nabla H(x_{k})\\
-\nabla H(x_{k})^{*} & 0
\end{array}\right)\left(\begin{array}{c}
\xi_{k}\\
\eta'(k)
\end{array}\right),\left(\begin{array}{c}
\rho_{k}\\
0
\end{array}\right)\right\rangle \right\rangle _{J}\\
 & =\left\langle \left\langle \rho_{k},\rho_{k}\right\rangle \right\rangle _{J},
\end{align*}
where $(*)$ used \eqref{eq:hessian}. Thus $\rho_{k}=0$ as desired.
This concludes the proof of \eqref{eq:to prove} and hence of Theorem
\ref{thm:rab index-1}.

\subsection{Proving that the index $\mu_{\textrm{Rab}}$ does not depend on the
choice of Hamiltonian}

\label{subsection:mu}

$\ $\vspace{6pt}

In this section we prove that the grading $\mu_{\textrm{Rab}}$ from
Definition \ref{defn:mu} does not depend on the choice of Hamiltonian
$H$ representing $\Sigma$. We also explain how is $\mu_{\textrm{Rab}}$
related to a suitably defined \textbf{transverse} Conley-Zehnder index.\\

We begin in a more general situation. Let $\Sigma\subseteq T^{*}M$
be a closed oriented separating hypersurface of $T^{*}M$, where the
latter is endowed with a twisted symplectic form $\omega$. Let $\mathcal{L}$
denote the characteristic line bundle of $\Sigma$. Note that $\mathcal{L}$
carries a natural orientation. Denote by $\mathcal{Q}$ the quotient
bundle $T\Sigma/\mathcal{L}$ with projection $p:T\Sigma\to\mathcal{Q}$.
Thus $\mathcal{Q}$ is a rank $2n-2$ bundle over $\Sigma$, which
carries the structure of a symplectic vector bundle. Let $T^{v}T^{*}M\subseteq TT^{*}M$
denote the vertical distribution defined by $T_{(q,p)}^{v}T^{*}M:=T_{(q,p)}T_{q}^{*}M$,
and denote by $\mathcal{V}\subseteq\mathcal{Q}$ the Lagrangian subbundle
of $\mathcal{Q}$ defined by $\mathcal{V}_{x}:=p\left(T_{x}^{v}T^{*}M\cap T_{x}\Sigma\right)$.

Suppose $y:\mathbb{R}/T\mathbb{Z}\rightarrow\Sigma$ is a smooth curve.
The bundle $y^{*}\mathcal{Q}$ over $\mathbb{R}/T\mathbb{Z}$ is trivial.
Let us say that a trivialization $\Psi:\mathbb{R}/T\mathbb{Z}\times\mathbb{R}^{2n-2}\rightarrow y^{*}\mathcal{Q}$
is \textbf{$\mathcal{V}$-preserving }if $\Psi(t)(V^{n-1})=\mathcal{V}_{y(t)}$
for each $t\in\mathbb{R}/T\mathbb{Z}$, where $V^{n-1}:=(0)\times\mathbb{R}^{n-1}$.
Such trivializations always exist; this is proved by first choosing
a trivialization $\Psi'$ of the trivial bundle $y^{*}\mathcal{V}$
and extending it to a trivialization $\Psi$ of $y^{*}\mathcal{Q}$
by making use of a compatible almost complex structure on $\mathcal{Q}$
to write $\mathcal{Q}=\mathcal{V}\oplus J\mathcal{V}$ (see for instance
\cite[Lemma 1.2]{AbbondandoloSchwarz2006}).

\begin{defn} A \textbf{defining Hamiltonian }for $\Sigma$ is an
autonomous Hamiltonian $H\in C^{\infty}(T^{*}M,\mathbb{R})$ such
that $\Sigma=H^{-1}(0)$ and $X_{H}|_{\Sigma}$ is a positively oriented
non-vanishing section of $\mathcal{L}$. Denote by $\mathcal{D}(\Sigma)$
the set of all defining Hamiltonians. Note that $\mathcal{D}(\Sigma)$
is a convex set. If $H\in\mathcal{D}(\Sigma)$ and $\phi_{t}^{H}$
denotes the flow of $X_{H}$ then $d\phi_{t}^{H}$ induces a linear
map $Q_{t}^{H}:\mathcal{Q}_{x}\rightarrow\mathcal{Q}_{\phi_{t}^{H}(x)}$
for each $x\in\Sigma$.\end{defn} 

\begin{rem} Our definition of a {}``defining Hamiltonian'' is weaker
than the usual one in Rabinowitz Floer homology (\cite{CieliebakFrauenfelder2009}).
There they ask that not only is $X_{H}|_{\Sigma}$ a positively oriented
non-vanishing section of $\mathcal{L}$, but that it is actually equal
to some given fixed section - in the contact case this is normally
the Reeb vector field. \end{rem} 

Note that if $H$ and $F$ are two elements of $\mathcal{D}(\Sigma)$
then there exists a smooth positive function $f\in C^{\infty}(\Sigma,\mathbb{R}^{+})$
such that $X_{H}|_{\Sigma}=fX_{F}|_{\Sigma}$. Then the Hamiltonian
flows $\phi_{t}^{H}|_{\Sigma}$ and $\phi_{t}^{F}|_{\Sigma}$ are
related by \[
\phi_{t}^{H}(x)=\phi_{\alpha(t,x)}^{F}(x)\ \ \ \mbox{for }x\in\Sigma.\]
 where \[
\alpha(t,x)=\int_{0}^{t}f(\phi_{s}^{H}(x))ds.\]
 Suppose $y:\mathbb{R}/T\mathbb{Z}\rightarrow\Sigma$ is a periodic
orbit of $X_{H}$. Let $\beta:[0,\alpha(T,y(0))]\rightarrow[0,T]$
be the inverse of the function $t\mapsto\alpha(t,y(0))$. Then the
curve $z:\mathbb{R}/(\alpha(T,y(0))\mathbb{Z}\rightarrow\Sigma$ defined
by $z(t)=y(\beta(t))$ is a periodic orbit of $X_{F}$. We will say
that $z$ is the orbit of $X_{F}$ corresponding to the orbit $y$
of $X_{H}$. \begin{defn} Let $H\in\mathcal{D}(\Sigma)$. Suppose
$y:\mathbb{R}/T\mathbb{Z}\rightarrow\Sigma$ is a periodic orbit of
$X_{H}$. Let $\Psi$ denote a $\mathcal{V}$-preserving trivialization
of $y^{*}\mathcal{Q}$, and denote by $\lambda_{y,\Psi}:[0,T]\rightarrow\mbox{Sp}(2n-2)$
the path \[
\lambda_{y,\Psi}(t):=\Psi(t)^{-1}\circ Q_{t}^{H}\circ\Psi(0).\]
 We denote by $\mu_{\textrm{CZ}}^{\tau}(y)$ the \textbf{transverse
Conley-Zehnder index} of $y$,\textbf{ }which by definition is given
by \[
\mu_{\textrm{CZ}}^{\tau}(y)=\mu_{\textrm{Ma}}(\mbox{graph}(\lambda_{y,\Psi}),\Delta_{\mathbb{R}^{2n-2}})\in\frac{1}{2}\mathbb{Z},\]
 where $\mu_{\textrm{Ma}}(\cdot,\cdot)$ denotes the \textbf{relative
Maslov index }as defined in \cite{RobbinSalamon1993} and $\Delta_{\mathbb{R}^{2n-2}}$
denotes the diagonal in $\mathbb{R}^{2n-2}$. Note however that our
sign conventions for $\mu_{\textrm{Ma}}(\cdot,\cdot)$ match those
of \cite{AbbondandoloPortaluriSchwarz2008} not \cite{RobbinSalamon1993}.
As the notation suggests, the index $\mu_{\textrm{CZ}}^{\tau}(y)$
is independent of the choice of $\mathcal{V}$-preserving trivialization
$\Psi$; this can be proved in exactly the same way as \cite[Lemma 1.3]{AbbondandoloSchwarz2006}.
\end{defn} The following lemma is immediate from the homotopy invariance
of $\mu_{\textrm{Ma}}(\cdot,\cdot)$ and the fact that $\mathcal{D}(\Sigma)$
is convex. 

\begin{lem} \label{lem:doesntdepend}Let $H,F\in\mathcal{D}(\Sigma)$.
Let $y$ denote a periodic orbit of $X_{H}$ contained in $\Sigma$,
and let $z$ denote the periodic orbit of $X_{F}$ corresponding to
$y$. Then \[
\mu_{\textrm{\emph{CZ}}}^{\tau}(y)=\mu_{\textrm{\emph{CZ}}}^{\tau}(z).\]

\end{lem} %
{}

Assume now that $H\in\mathcal{D}(\Sigma)$ and $y$ is a periodic
orbit of $X_{H}$ contained in $\Sigma$ admitting an orbit cylinder
$\mathcal{O}=(y_{s})_{s\in(-\varepsilon,\varepsilon)}$. Associated
to such an orbit cylinder $\mathcal{O}$ is the map \[
\xi(t):=\frac{\partial}{\partial s}\Bigl|_{s=0}y_{s}(t).\]
 By differentiating the equation $H(y_{s})\equiv s$ we see that \[
d_{y(t)}H(\xi(t))\equiv1.\]
The tangent space $T_{y(t)}\mathcal{O}$ is spanned by $X_{H}(y(t))$
and $\xi(t)$, and forms a symplectic vector subspace of $T_{y(t)}T^{*}M$.
Moreover by differentiating the equation\[
\phi_{T(s)}^{H}(y_{s}(0))=y_{s}(T(s))=y_{s}(0)\]
with respect to $s$ and setting $s=0$ we obtain
\begin{equation}
\xi(T)+T'(0)X_{H}(y(T))=\xi(0).\label{eq:why it is a triv}\end{equation}
Also, by differentiating the equation\[
\phi_{t}^{H}(y_{s}(0))=y_{s}(t)\]
with respect to $s$ and setting $s=0$ we obtain
\begin{equation}
d_{y(0)}\phi_{t}^{H}(\xi(0))=\xi(t).\label{eq:sym map}\end{equation}

Let us denote by $E_{y(t)}:=(T_{y(t)}\mathcal{O})^{\perp}$ the $\omega$-orthogonal
complement of $T_{y(t)}\mathcal{O}$ in $T_{y(t)}T^{*}M$. We denote
by $y^{*}T\mathcal{O}$ and $y^{*}E$ the induced bundles over $\mathbb{R}/T\mathbb{Z}$.
The following lemma is clear. \begin{lem} There is a natural isomorphism
of symplectic vector bundles over $\mathbb{R}/T\mathbb{Z}$: \[
y^{*}\mathcal{Q}\cong y^{*}E.\]

\end{lem} 

Define a trivialization $\Phi:\mathbb{R}/T\mathbb{Z}\times\mathbb{R}^{2}\rightarrow y^{*}T\mathcal{O}$
by \[
\Phi(t)(0,1):=X_{H}(y(t)),\ \ \ \Phi(t)(1,0):=\xi(t)+\frac{tT'(0)}{T}X_{H}(y(t))\]
(equation \eqref{eq:why it is a triv} shows that this is well defined,
i.e. $\Phi(T)=\Phi(0)$). Using \eqref{eq:sym map} and the fact that
$d_{y(0)}\phi_{t}^{H}(X_{H}(y(0)))=X_{H}(y(t))$, we deduce that the
curve $\lambda_{y,\Phi}:[0,T]\rightarrow\mbox{Sp}(2)$ defined by
\[
\lambda_{y,\Phi}(t):=\Phi(t)^{-1}\circ d_{y(0)}\phi_{t}^{H}|_{T\mathcal{O}}\circ\Phi(0)\]
 is given by the matrix \[
\left(\begin{array}{cc}
1 & 0\\
-\frac{tT'(0)}{T} & 1
\end{array}\right).\]
By a well-known computation
(see for instance \cite[Lemma 4.3]{CieliebakFrauenfelder2009}):
\begin{equation}
\mu_{\textrm{Ma}}(\mbox{graph}(\lambda_{y,\Phi}),\Delta_{\mathbb{R}^{2}})=\begin{cases}
\frac{1}{2}\mbox{sign}(-T'(0)) & T'(0)\ne0\\
0 & T'(0)=0.
\end{cases}\label{eq:sign T dash}\end{equation}
 We can identify $\mathcal{V}$-preserving trivializations of $y^{*}\mathcal{Q}$
with trivializations $\Psi:\mathbb{R}/T\mathbb{Z}\times\mathbb{R}^{2n-2}\rightarrow y^{*}E$
with the property that \[
\Psi(t)(V^{n-1})=E_{y(t)}\cap((T_{y(t)}\Sigma\cap T_{y(t)}^{v}T^{*}M)\oplus\mathcal{L}_{y(t)}).\]
 Thus if $\Psi$ is a such a trivialization and $\lambda_{y,\Psi}:[0,T]\rightarrow\mbox{Sp}(2n-2)$
denotes the curve defined by \[
\lambda_{y,\Psi}(t):=\Psi(t)^{-1}\circ d_{y(0)}\phi_{t}^{H}|_{E_{y(t)}}\circ\Psi(0).\]
 then \[
\mu_{\textrm{CZ}}^{\tau}(y)=\mu_{\textrm{Ma}}(\mbox{graph}(\lambda_{y,\Psi}),\Delta_{\mathbb{R}^{2n-2}}).\]
 We can glue the two trivializations $\Phi$ and $\Psi$ together
to get a trivialization $\Phi\oplus\Psi$ of the full pullback bundle
$y^{*}TT^{*}M$. This trivialization has the property that \begin{equation}
(\Phi\oplus\Psi)(t)(V^{n})=(T_{y(t)}\Sigma\cap T_{y(t)}^{v}T^{*}M)\oplus\mathcal{L}_{y(t)}.\label{eq:v tilde}\end{equation}
 As before let $\lambda_{y,\Phi\oplus\Psi}:[0,T]\rightarrow\mbox{Sp}(2n)$
denote the curve defined by \[
\lambda_{y,\Phi\oplus\Psi}(t):=(\Phi\oplus\Psi)(t)^{-1}\circ d_{y(0)}\phi_{t}^{H}\circ(\Phi\oplus\Psi)(0).\]
 Let us denote by \[
\widetilde{\mu}_{\textrm{CZ}}(y):=\mu_{\textrm{Ma}}(\mbox{graph}(\lambda_{\Phi\oplus\Psi}),\Delta_{\mathbb{R}^{2n}})\in\frac{1}{2}\mathbb{Z}.\]
 This half-integer is independent of the trivialization $\Phi\oplus\Psi$
in the sense that if $\Theta$ is any other trivialization of $y^{*}TT^{*}M$
satisfying \eqref{eq:v tilde} then \[
\mu_{\textrm{Ma}}(\mbox{graph}(\lambda_{\Theta}),\Delta_{\mathbb{R}^{2n}})=\mu_{\textrm{Ma}}(\mbox{graph}(\lambda_{\Phi\oplus\Psi}),\Delta_{\mathbb{R}^{2n}}).\]

Moreover we have by \eqref{eq:sign T dash} and the product axiom
of $\mu_{\textrm{Ma}}(\cdot,\cdot)$ (\cite[Theorem 4.1]{RobbinSalamon1993})
that the following holds. \begin{lem} Let $y$ be a periodic orbit
of $X_{H}$ contained in $\Sigma$ admitting an orbit cylinder. Then
\[
\widetilde{\mu}_{\textrm{\emph{CZ}}}(y)=\mu_{\textrm{\emph{CZ}}}^{\tau}(y)+\frac{1}{2}\mbox{\emph{sign}}(-T'(0)).\]

\end{lem}

In the statement of the lemma, if $T'(0)=0$ we understand that $\mbox{{\rm {sign}}}(-T'(0))=0$.

We now relate $\widetilde{\mu}_{\textrm{CZ}}(y)$ with the \textbf{Conley-Zehnder
index }$\mu_{\textrm{CZ}}(y)$. By definition \[
\mu_{\textrm{CZ}}(y)=\mu_{\textrm{Ma}}(\mbox{graph}(\lambda_{\Omega}),\Delta_{\mathbb{R}^{2n}})\]
 where $\Omega:\mathbb{R}/T\mathbb{Z}\times\mathbb{R}^{2n}\rightarrow y^{*}TT^{*}M$
is any \textbf{vertical preserving trivialization}, that is for all
$t\in\mathbb{R}/T\mathbb{Z}$,\begin{equation}
\Omega(t)(V^{n})=T_{y(t)}^{v}T^{*}M,\label{eq:vertical preserving}\end{equation}
 and $\lambda_{\Omega}:[0,T]\rightarrow\mbox{Sp}(2n)$ is defined
analogously to before. As before $\mu_{\textrm{CZ}}(y)$ is independent
of the vertical preserving trivialization $\Omega$.

Let us denote by $m:\pi_{1}(\mbox{Sp}(2n))\rightarrow\mathbb{Z}$
is the classical Maslov index. See for instance \cite[Theorem 2.29]{McDuffSalamon1998}
for the definition of $m$. The loop axion of the Conley-Zehnder index
(see \cite[Section 2.4]{Salamon1999}) together with the product axiom
of $m$ implies that the difference of the indices is given by \[
\widetilde{\mu}_{\textrm{CZ}}(y)-\mu_{\textrm{CZ}}(y)=2m(\Theta^{-1}\circ\Omega),\]
 where $\Theta$ (resp. $\Omega$) is any trivialization of $y^{*}TT^{*}M$
satisfying \eqref{eq:v tilde} (resp. \eqref{eq:vertical preserving}).
This integer visibly depends only on $\Sigma$ and the free homotopy
class of $y$; in fact the function $[y]\mapsto2m(\Theta^{-1}\circ\Omega)$
defines a class $c(\Sigma)\in H^{1}(\Sigma,\mathbb{Z})$.

\begin{cor} \label{indep}Suppose $H\in\mathcal{D}(\Sigma)$ and
$y$ is a periodic orbit of $X_{H}$ contained in $\Sigma$ which
admits an orbit cylinder. Then \[
\mu_{\textrm{\emph{CZ}}}(y)-\frac{1}{2}\mbox{\emph{sign}}(-T'(0))=\mu_{\textrm{\emph{CZ}}}^{\tau}(y)-c(\Sigma)([y]).\]
 The right-hand side is independent of the choice of $H\in\mathcal{D}(\Sigma)$
by Lemma \ref{lem:doesntdepend} (and the right-hand side is defined
even when $y$ does not admit an orbit cylinder), and hence so is
the left-hand side.\\
 \end{cor}

In the case we will be most interested in $c(\Sigma)$ is always zero.
\begin{lem} Assume that $\Sigma$ is transverse to each fibre $T_{q}^{*}M$
along a closed curve $y:S^{1}\to\Sigma$. Then $c(\Sigma)([y])=0$.
\label{lem:trans} \end{lem} \begin{proof}

By \cite[Remark 5.3]{RobbinSalamon1993} one has \[
2m(\Theta^{-1}\circ\Omega)=\mu_{\textrm{Ma}}(\Theta^{-1}\circ\Omega(V^{n}),V^{n}).\]
 If $\Sigma$ is transverse to $T^{v}T^{*}M$ then the path $\Theta(t)^{-1}\circ\Omega(t)$
takes values in $\mbox{Sp}_{n-1}(2n)\subseteq\mbox{Sp}(2n)$, that
is,\[
\dim\,\Theta(t)^{-1}\circ\Omega(t)(V^{n})\cap V^{n}=n-1\ \ \ \mbox{for all }t\in[0,1].\]
 Thus by the zero axiom of $\mu_{\textrm{Ma}}(\cdot,\cdot)$ (see
\cite[Theorem 4.1]{RobbinSalamon1993}) $\mu_{\textrm{Ma}}(\Theta^{-1}\circ\Omega(V^{n}),V^{n})=0$.\end{proof}

Let us now revert back to the case studied in this note, where $\Sigma=H^{-1}(k)$
for $H$ a Tonelli Hamiltonian. The previous corollary proves that
the grading $\mu_{\textrm{Rab}}$ from Definition \ref{defn:mu} depends
only on $\Sigma$ and not on $H$. In fact in this case $c(\Sigma)=0$
always. The points $(q,p)\in\Sigma$ where $\Sigma$ is not transverse
to $T_{q}^{*}M$ have the form $(q,\mathfrak{L}(q,0))$ (where $\mathfrak{L}$
is the Legendre transform, see \eqref{eq:legendre transform}) and
they form a submanifold $\Sigma_{0}$ of $\Sigma$ of codimension
$n$ if not empty. For $n\geq2$, any closed curve $y:S^{1}\to\Sigma$
can be slightly deformed in $\Sigma$ to a curve that misses $\Sigma_{0}$
and hence from Lemma \ref{lem:trans} we conclude that $c(\Sigma)([y])=0$.
Thus for the class of systems considered in this paper, if $(x,\eta)\in\mbox{Crit}(\mathcal{A}_{k})$
with $\eta>0$, and $y(t):=x(t/\eta)$ then: \[
\mu_{\textrm{Rab}}(x,\eta)=\mu_{\textrm{CZ}}(y)-\frac{1}{2}\mbox{sign}(-T'(0))=\mu_{\textrm{CZ}}^{\tau}(y).\]

\section{An example}

In this section we construct an example of a Hamiltonian $H$ on $T^{*}S^{2}$
with a non-degenerate periodic orbit $y$ of $X_{H}$ with energy
$k>c$ for which the correction term $\chi(y)=-1$. Let $(r,\theta)$
denote polar coordinates on $\mathbb{R}^{2}$, and let $D:=\{(r,\theta)\,:\, r\in(0,4)\}$.
Let $f:(0,\infty)\rightarrow\mathbb{R}$ denote a smooth function
such that $f\equiv0$ on $(0,1]$ and $[3,\infty)$, which will be
specified precisely later. Define $L:TD\rightarrow\mathbb{R}$ by
\[
L(r,\theta,\dot{r},\dot{\theta})=\frac{1}{2}(\dot{r}^{2}+r^{2}\dot{\theta}^{2})-f(r)\dot{\theta}.\]
Let $E:TD\rightarrow\mathbb{R}$ be defined by $E(r,\theta,\dot{r},\dot{\theta})=\frac{1}{2}(\dot{r}^{2}+r^{2}\dot{\theta}^{2})$.

$L$ is the Lagrangian associated to the flat metric and the exact
magnetic form $\sigma=d(-fd\theta)=-f'(r)dr\wedge d\theta$. Let $\sigma_{\textrm{area}}$
denote the area form, so $\sigma_{\textrm{area}}=rdr\wedge d\theta$.
Then $\sigma=-\frac{f'(r)}{r}\sigma_{\textrm{area}}$. Let $F:D\rightarrow\mathbb{R}$
be defined by $F(r,\theta):=-f'(r)/r$. Let $\mathtt{i}:TD\rightarrow TD$
denote rotation by $+\pi/2$.

Suppose $\gamma(t)=(r(t),\theta(t))$. Then $\gamma$ satisfies the
Euler-Lagrange equations for $L$ if and only if \[
\ddot{r}=\dot{\theta}(r\dot{\theta}-f'(r)),\ \ \ r^{2}\dot{\theta}-f(r)=\mbox{const.}\]
Suppose $\gamma:\mathbb{R}/T\mathbb{Z}\rightarrow D$ is a curve such
that $(\gamma,\dot{\gamma})$ is a closed orbit of the Euler-Lagrange
flow of $L$ contained in $E^{-1}(k)$. Let $(\gamma_{s}:\mathbb{R}/T\mathbb{Z}\rightarrow D)_{s\in(-\varepsilon,\varepsilon)}$
denote a variation along $\gamma$ such that $(\gamma_{s},\dot{\gamma}_{s})\subseteq E^{-1}(k)$
for each $s\in(-\varepsilon,\varepsilon)$, and denote by $\zeta=\frac{\partial}{\partial s}\bigl|_{s=0}\gamma_{s}$
the associated Jacobi field. Since $\{\dot{\gamma}(t),\mathtt{i}\dot{\gamma}(t)\}$
constitutes an orthogonal frame along $\gamma$, there exist unique
functions $x,y:\mathbb{R}/T\mathbb{Z}\rightarrow\mathbb{R}$ such
that \[
\zeta(t)=x(t)\dot{\gamma}(t)+y(t)\mathtt{i}\dot{\gamma}.\]
Set \[
K(\sigma):=\left\langle \nabla F,\mathtt{i}\dot{\gamma}\right\rangle +F^{2}.\]
Then the functions $x$ and $y$ satisfy the \textbf{Jacobi equations}
\[
\dot{x}+Fy=0.\]
\[
\ddot{y}+K(\sigma)y=0\]
(see for instance \cite[p495]{PaternainPaternain1997a}). Provided
$K(\sigma)>0$, we claim that if \[
\frac{2\pi}{T\sqrt{K(\sigma)}}\notin\mathbb{Q}\]
then the orbit $(\gamma,\dot{\gamma})$ is weakly non-degenerate,
that is, the space of periodic Jacobi fields (with initial conditions
tangent to the energy level) along $\gamma$ is $1$-dimensional,
spanned by $\dot{\gamma}$. Indeed, if a Jacobi field $\zeta=x\dot{\gamma}+y\mathtt{i}\dot{\gamma}$
existed with $y$ not identically zero then the period of $y$ must
be commensurable with $T$, which is equivalent to asking that $\frac{2\pi}{T\sqrt{K(\sigma)}}\in\mathbb{Q}$.
(This latter condition will imply in fact that all iterates of $\gamma$
are non-degenerate.)

Now let us suppose that \[
\gamma_{k}(t)=(r_{k}(t),\theta_{k}(t)):\mathbb{R}/T(k)\mathbb{Z}\rightarrow D\]
 is a loop in $D$ such that $(\gamma_{k},\dot{\gamma}_{k})$ is an
orbit of the Euler-Lagrange flow of $L$ contained in $E^{-1}(k)$.
Let us suppose moreover that \[
r_{k}(t)\equiv\rho(k),\ \ \ \theta_{k}(t)=a(k)t\]
 for some constants $a(k)>0$ and $\rho(k)\in(0,4)$.

For such a curve $\gamma_{k}$ to be an orbit we need $\rho(k)a(k)-f'(\rho(k))=0$,
and in order to have energy $k$ we need $\rho(k)^{2}a(k)^{2}=2k$.
Thus \begin{equation}
\sqrt{2k}=\rho(k)a(k)=f'(\rho(k)).\label{eq:conditions on k}\end{equation}

Thus \[
T(k)=\frac{2\pi}{a(k)}=\frac{2\pi\rho(k)}{\sqrt{2k}}=\frac{2\pi\rho(k)}{f'(\rho(k))}.\]
 With this choice of $\gamma_{k}$ we have $\dot{\gamma}_{k}=(0,a(k))$
and $\left|\dot{\gamma}_{k}\right|=\rho(k)a(k)$. Thus $\mathtt{i}\dot{\gamma}_{k}=(-\rho(k)a(k),0)$,
and thus \[
K(\sigma)=-\rho(k)a(k)F'(\rho(k))+F(\rho(k))^{2}.\]
 Substituting $F=-f'(r)/r$ and simplifying we obtain \[
K(\sigma)=a(k)f''(\rho(k)).\]
 Let us now focus on $k=1/2$ and choose as our initial condition
to define $\gamma$ that \[
a(1/2)=1/2.\]
 Then \eqref{eq:conditions on k} implies that \[
\rho(1/2)=2,\ \ \ f'(\rho(1/2))=f'(2)=1.\]
 Then $\sqrt{K(\sigma)}=\sqrt{\frac{f''(2)}{2}}$, and thus provided
$f''(2)>0$, we see that in order for $(\gamma,\dot{\gamma})$ to
be weakly non-degenerate suffices to have \[
\frac{2\pi}{T(k)\sqrt{K(\sigma)}}=\frac{2\pi}{4\pi\sqrt{\frac{f''(2)}{2}}}=\frac{1}{\sqrt{2f''(2)}}\notin\mathbb{Q}.\]
 If in addition $T'(1/2)\neq0$, then the orbit will also be strongly
non-degenerate. Since $T(k)=\frac{2\pi\rho(k)}{f'(\rho(k))}$, we
compute that\[
T'(k)=\frac{f'(\rho(k))2\pi\rho'(k)-2\pi\rho(k)f''(\rho(k))\rho'(k)}{f'(\rho(k))^{2}},\]
 and hence \[
T'(1/2)=2\pi\rho'(1/2)(1-2f''(2)).\]
 Note also from \eqref{eq:conditions on k} we have \[
\frac{1}{\sqrt{k}}=f''(\rho(k))\rho'(k),\]
 and hence if $0<f''(2)<1/2$ then we have $\rho'(1/2)>0$, and hence
also $T'(1/2)>0$.

Since we have chosen $f$ such that $f$ vanishes outside $\{1<r<3\}$,
we can embed $D$ as a subset of $S^{2}$ and find a $1$-form $\psi$
on $S^{2}$ such that $\psi$ coincides with $-fd\theta$ on $D$
and vanishes on $S^{2}\backslash D$. Let $g$ denote a metric on
$S^{2}$ that restricts to $D$ to define the standard flat metric,
and let $H:T^{*}S^{2}\rightarrow\mathbb{R}$ be defined by $H(q,p)=\frac{1}{2}\left|p\right|_{g}^{2}$.
Set $\Sigma_{k}:=H^{-1}(k)$. The Ma\~n\'e critical value $c=c_{0}$
of $H$ can be estimated by \[
c\leq\sup_{q\in S^{2}}\frac{1}{2}\left|\psi_{q}\right|^{2}=\sup_{r\in(1,3)}\frac{1}{2}\frac{\left|f(r)\right|^{2}}{r^{2}}\leq\frac{1}{2}\sup_{r\in(1,3)}\left|f(r)\right|^{2}\]
(see equations \eqref{eq:mcv1} and \eqref{eq:mcv2}). Putting this
together, suppose we choose our function $f:(0,\infty)\rightarrow\mathbb{R}$
such that: 
\begin{enumerate}
\item $f\equiv0$ on $(0,\infty)\backslash(1,3)$ and $0\leq f(r)<1$ for
all $r\in(0,\infty)$; 
\item $f'(2)=1$; 
\item $0<f''(2)<1/2$ and $1/\sqrt{2f''(2)}\notin\mathbb{Q}$. 
\end{enumerate}
Such functions $f$ clearly exist. Let $\omega:=dp\wedge dq+\pi^{*}(d\psi)$,
and write as usual $X_{H}$ for the symplectic gradient of $H$ with
respect to $\omega$. With this choice of $f$, we have shown that: 
\begin{enumerate}
\item $c<1/2$. 
\item There exists a closed non-degenerate orbit $y:\mathbb{R}/T\mathbb{Z}\rightarrow\Sigma_{1/2}$
of $X_{H}$ with $\chi(y)\overset{\textrm{def}}{=}\mbox{sign}(-T'(1/2))<0$. 
\end{enumerate}
Finally if we consider a new orbit $\alpha$ with initial condition
$a(1/2)=2/5$ then a similar computation shows that if in addition
we ask that $f'(5/2)=1$, $f''(5/2)>2/5$ and $\sqrt{2/5f''(5/2)}\notin\mathbb{Q}$
then the corresponding orbit $z$ of $X_{H}$ has $\chi(z)>0$.

\bibliographystyle{amsplain}
\bibliography{C:/Users/Will/Desktop/willbibtex}

\end{document}